\documentclass[11pt]{article}

\usepackage{amsmath,amsthm,amssymb,graphicx}
\usepackage{amsfonts}
\usepackage{mathabx}

\usepackage[margin=0.97in,top=.99in,bottom=1.1in, textheight=22.51cm,textwidth=16.51cm, heightrounded,columnsep=0.81cm]{geometry}

\usepackage{wrapfig}

\pdfoutput=1  

\usepackage[usenames,dvipsnames]{color}
\usepackage[colorlinks,bookmarksopen,bookmarksnumbered,citecolor=BrickRed,urlcolor=RoyalPurple,linkcolor=BrickRed]{hyperref}

\newlength{\noteWidth}
\long\def\notes#1{\ifinner
           {\footnotesize #1}
           \else
           \marginpar{\parbox[t]{\noteWidth}{\raggedright\footnotesize #1}}
       \fi\typeout{#1}}

\def\notes#1{\typeout{read notes: #1}}  

\def\urls#1{{\small \url{#1}}}

\def\util{{\mathcal{U}}}

\def\meanutil{\mbox{\scriptsize$\widebar{\cal U}$}}

\def\cKL{c_{\text{\tiny KL}}}

\def\reward{w}

\makeatletter
\newdimen\rh@wd
\newdimen\rh@hta
\newdimen\rh@htb
\newbox\rh@box
\def\rh@measure#1{\setbox\rh@box=\hbox{$#1$}\rh@wd=\wd\rh@box \rh@hta=\ht\rh@box}

\def\widecheck#1{\rh@measure{#1}%
  \setbox\rh@box=\hbox{$\widehat{\vrule height \rh@hta width\z@ \kern\rh@wd}$}%
  \rh@htb=\ht\rh@box \advance\rh@htb\rh@hta \advance\rh@htb\p@
  \ooalign{$\vrule height \ht\rh@box width\z@ #1$\cr
           \raise\rh@htb\hbox{\scalebox{1}[-1]{\box\rh@box}}\cr}}
\makeatother

\def\xz{x^\circ}

\def\EFn#1{\Lambda_{#1}}

\def\preclH{{\cal H}^\circ}
\def\preH{H^\circ}  

\def\sq{\hbox{\rlap{$\sqcap$}$\sqcup$}}
\def\qed{\ifmmode\sq\else{\unskip\nobreak\hfil
\penalty50\hskip1em\null\nobreak\hfil\sq
\parfillskip=0pt\finalhyphendemerits=0\endgraf}\fi\medskip}

\long\def\defbox#1{\framebox[.9\hsize][c]{\parbox{.85\hsize}{%
\parindent=0pt
\baselineskip=12pt plus .1pt      
\parskip=6pt plus 1.5pt minus 1pt 
 #1}}}

\long\def\beginbox#1\endbox{\subsection*{}%
\hbox{\hspace{.05\hsize}\defbox{\medskip#1\bigskip}}%
\subsection*{}}

\def\endbox{}

\newsavebox{\junk}
\savebox{\junk}[1.6mm]{\hbox{$|\!|\!|$}}

\def\liminf{\mathop{\rm lim\ inf}}

\def\argmax{\mathop{\rm arg\, max}}

\def\U{{\sf U}}

\def\state{{\sf X}}
\def\stateu{{\sf X}_{\sf u}}
\def\staten{{\sf X}_{\sf n}}

\newcommand{\field}[1]{\mathbb{#1}}

\def\Re{\field{R}}

\def\ind{\field{I}}

\def\ZZ{\field{Z}}

\def\cpi{\check{\pi}} 
\def\cP{{\check{P}}}
\def\cZ{{\check{Z}}}

\def\bfmath#1{{\mathchoice{\mbox{\boldmath$#1$}}%
{\mbox{\boldmath$#1$}}%
{\mbox{\boldmath$\scriptstyle#1$}}%
{\mbox{\boldmath$\scriptscriptstyle#1$}}}}

\def\bfmN{\bfmath{N}}

\def\bfmU{\bfmath{U}}

\def\haclW{\widehat{\cal W}}

\def\bfmX{\bfmath{X}}

\def\bfmY{\bfmath{Y}}

\def\bfmhhaY{\bfmath{\hhaY}} 
\def\bfmhhaY{\hbox to 0pt{$\widehat{\bfmY}$\hss}\widehat{\phantom{\raise 1.25pt\hbox{$\bfmY$}}}}

\def\bfmW{\bfmath{W}}

\def\haP{{\widehat P}}

\def\til={{\widetilde =}}

\def\clV{{\cal V}}
\def\clW{{\cal W}}

 \def\FRAC#1#2#3{\genfrac{}{}{}{#1}{#2}{#3}}

\def\ddtp{{\mathchoice{\FRAC{1}{d^{\hbox to 2pt{\rm\tiny +\hss}}}{dt}}%
{\FRAC{1}{d^{\hbox to 2pt{\rm\tiny +\hss}}}{dt}}%
{\FRAC{3}{d^{\hbox to 2pt{\rm\tiny +\hss}}}{dt}}%
{\FRAC{3}{d^{\hbox to 2pt{\rm\tiny +\hss}}}{dt}}}}

\def\ddzeta{{\mathchoice{\FRAC{1}{d}{d\zeta}}%
{\FRAC{1}{d}{d\zeta}}%
{\FRAC{3}{d}{d\zeta}}%
{\FRAC{3}{d}{d\zeta}}}}

\def\ddzetap{{\mathchoice{\FRAC{1}{d^+}{d\zeta}}%
{\FRAC{1}{d^+}{d\zeta}}%
{\FRAC{3}{d^+}{d\zeta}}%
{\FRAC{3}{d^+}{d\zeta}}}}

\def\half{{\mathchoice{\FRAC{1}{1}{2}}%
{\FRAC{1}{1}{2}}%
{\FRAC{3}{1}{2}}%
{\FRAC{3}{1}{2}}}}

\def\eqdef{\mathbin{:=}}

\def\Prob{{\sf P}}

\def\Expect{{\sf E}}

\def\average#1,#2,{{1\over #2} \sum_{#1}^{#2}}

\def\eye(#1){{\bf(#1)}\quad}

\def\varble{\,\cdot\,}

\newtheorem{theorem}{Theorem}[section]

\newtheorem{proposition}[theorem]{Proposition}

\def\Prop#1{Prop.~\ref{#1}}
\def\Theorem#1{Theorem~\ref{#1}}

\def\Section#1{Section~\ref{#1}}

\newcounter{rmnumx}
\newenvironment{romannumx}{\begin{list}{{\upshape (\roman{rmnumx})}}{\usecounter{rmnumx}
\setlength{\leftmargin}{2pt}
\setlength{\rightmargin}{4pt}
\setlength{\itemsep}{3pt}
\setlength{\itemindent}{18pt}
}}{\end{list}}

\newcounter{anum}

\graphicspath{{figures/}}

\def\Ebox#1#2{%
\begin{center}
\includegraphics[width= #1\hsize]{#2} \end{center}}

\def\Fig#1{Fig.~\ref{#1}}

\def\ind{\field{I}}

\def\Re{\field{R}}

  \title{Action-Constrained Markov Decision Processes 
  \\
  With Kullback--Leibler Cost\thanks{Funding from the ANR under grant ANR-16-CE05-0008, and NSF under awards EPCN~1609131, CPS~1646229 is gratefully acknowledged.}
}
\usepackage{times}

\author{Ana Bu\v{s}i\'{c}\thanks{Inria and DI ENS, \'Ecole normale sup\'erieure, CNRS, PSL Research University, Paris, France
    ({ana.busic@inria.fr}, \url{http://www.di.ens.fr/\string~busic/}).}
  \and
  Sean Meyn\thanks{Department of Electrical and Computer
Engineering at the University of Florida, Gainesville ({meyn@ece.ufl.edu}, \url{http://www.meyn.ece.ufl.edu/}).} 
}

\begin{document}

\maketitle

\begin{abstract}

This paper concerns computation of optimal policies in which the one-step reward function contains a cost term that models Kullback-Leibler divergence with respect to nominal dynamics. This technique was introduced by Todorov in 2007, where it was shown under general conditions that the solution to the average-reward optimality equations reduce to a simple eigenvector problem. Since then many authors have sought to apply this technique to control problems and models of bounded rationality in economics. 

A crucial assumption is that the input process is essentially unconstrained. For example, if the nominal dynamics include randomness from nature (e.g., the impact of wind on a moving vehicle), then the optimal control solution does not respect the exogenous nature of this disturbance. 

This paper introduces a technique to solve a more general class of action-constrained MDPs. The main idea is to solve an entire parameterized family of MDPs, in which the parameter is a scalar weighting the one-step  reward function. The approach is new and practical even in the original unconstrained formulation.

%
%
%
%
%
%

\end{abstract}

\smallbreak


\section{Introduction}

\label{s:intro}

Consider a Markov Decision Process (MDP) with finite state space $\state$, general action space $\U$,  and  one-step reward  function $\reward\colon\state\times\U\to\Re$.   Two standard optimal control criteria are \textit{finite-horizon}:
\begin{equation}
\clW^*_T(x) = \max \sum_{t=0}^{T}  \Expect[  \reward(X(t),U(t))  \mid X(0)=x ]
\label{e:TCobjectiveReward_intro}
\end{equation}
where $T\ge 0$ is fixed,   and \textit{average reward}:
\begin{equation} 
\eta^*(x) =\max\Bigl\{ \liminf_{T\to\infty}\frac{1}{T}   \sum_{t=0}^{T-1}  \Expect[  \reward(X(t),U(t))    \mid X(0)=x ]  \Bigr\}\,.
\label{e:ARobjectiveReward_intro}
\end{equation}
where $\bfmX=\{X(t): t\ge 0\}$, $\bfmU=\{U(t): t\ge 0\}$ denote the state and input sequences.

In either case, the maximum is over all  admissible input sequences; it is obtained as deterministic state feedback under general conditions. In the average-reward framework the optimal policy is typically stationary:  $U(t) = \phi^*(X(t))$ for a mapping $\phi^*\colon\state\to\U$,  and $\eta^*(x) $ does not depend upon the initial condition $x$   (see \cite{put14,bershr96a}).   

A special class of MDP models was introduced by~\cite{tod07}, for which either   optimal control problem has an attractive solution.  The reward   function is assumed to be the sum of two terms: 
\begin{equation*}
\begin{aligned}
\reward(x,\mu)  &= 
\util(x) 
		-  D(\mu\| P_0(x,\varble)).
\end{aligned}
\label{e:AROEcostP}
\end{equation*} 
The first term is a function   $\util\colon\state\to\Re$ that is completely unstructured.   The second is a ``control cost'',  defined using   Kullback--Leibler  (K-L) divergence (also known as \textit{relative entropy}).   The control cost is based on deviation from nominal (control-free) behavior;  modeled by a nominal transition matrix $P_0$:
\begin{equation*}
D(\mu\| P_0(x,\varble))\eqdef
\sum_{x'}  \mu(x') \log \Bigl(\frac{\mu(x') }{P_0(x,x')} \Bigr).
\end{equation*}

It is shown that the solution with respect to the average reward criterion is obtained as the solution to the following eigenvector problem:  let $(\lambda,v)$ denote the Perron-Frobenius eigenvalue-eigenvector pair for the positive matrix with entries  $\haP(x,x') = \exp(\util(x)) P_0(x,x')$,  $x,x'\in\state$.
The eigenvector property
$
\haP v = \lambda v
$ implies that the ``twisted'' matrix 
\begin{equation}
\cP(x,x') 
=
\frac{1}{\lambda}
\frac{v(x')}{v(x)}
\haP(x,x')  \,, \quad x,x'\in\state\, .
\label{e:twistedP}
\end{equation}
is a transition matrix on $\state$.  This transition matrix defines the dynamics of the model under optimal control.  A similar model was introduced in the earlier work of  \cite{kar96}, but without the complete solution reviewed here. 

%

\begin{wrapfigure}{l}{.28\hsize}
\vspace{-.5cm}
\Ebox{.95}{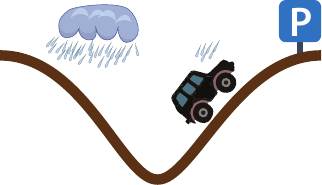}
\vspace{-.5cm}\caption{\small \small
Optimal hill climb}
\label{f:TodCar} 
\vspace{-.25cm}
\end{wrapfigure}

Since the publication of \cite{tod07} there has been significant theoretical advancement, with proposed applications to economics \cite{guaragwil14}, distributed control \cite{meybarbusyueehr15}, and neuroscience \cite{doy09}.

It is appealing to imagine that rational economic agents are solving an eigenvector problem to maximize their utility.  However, a careful look at the controlled dynamics \eqref{e:twistedP} suggests a limitation of this MDP formulation:  \textit{how can this transformation respect exogenous disturbances from nature?}  
An essential assumption in this prior work is that for each $x$,  and any pmf $\mu$,  it is possible to choose the action so that $P(x,x') = \mu(x')$.  This is equivalent to the assumption that  the action space 
$\U$ consists of all probability mass functions on $\state$, and the controlled transition matrix is entirely determined by the input as follows:
\begin{equation}
\Prob\{X(t+1)=x'\mid X(t) =x, U(t) = \mu\} = \mu(x')\,,\qquad x,x'\in\state,\ \mu\in\U\, .
\label{e:TodMDPu}
\end{equation}
This modeling assumption presents a significant limitation, as pointed out in \cite{tod09}:  ``\textit{It prevents us from modeling systems subject to disturbances outside the actuation space}''.

\Fig{f:TodCar} is based on an example of \cite{tod09}.  Reaching the parking spot at the top of the hill in minimum time (or  minimal fuel) is formulated as a total \textit{cost} problem, similar to \eqref{e:TCobjectiveReward_intro}.   
The figure has been modified to indicate that wind and rain influence the behavior of the car on the track.    The optimal solution cannot take the form \eqref{e:twistedP} when this additional randomness is included in the model, since this would mean our control action would modify the weather.

\paragraph{Contributions}

In this paper the K-L cost framework is broadened to include constraints on the pmf $\mu$ appearing in \eqref{e:TodMDPu}.  
The new approach to computation   is based on the solution of an entire family of MDP problems,  parameterized by a scalar $\zeta$ appearing as a weighting factor in the one-step reward function.  Letting $X(t)$ denote the state, and $R(t)$ denote the randomized policy at time $t$, this one-step reward is of the form
\begin{equation}
\reward(X(t),R(t)) = \zeta \util(X(t)) -  \cKL(X(t),R(t) ) 
\label{e:oneStepReward}
\end{equation}
in which $\cKL$ denotes relative entropy with respect to nominal dynamics (see \eqref{e:cKL}).   

The main results of the paper are contained in Theorems~\ref{t:TCODE}  and \ref{t:IDPODE}, 
with parallel results for the total- and average-reward control problems.   In each case, it is shown that   \textit{the solution to an entire family of MDPs can be obtained through the solution of a single ordinary differential equation} (ODE).   

The ODE solution is most elegant in the average-reward setting.  For each $\zeta$, the solution to the average-reward optimization problem is based on a relative value function $h^*_\zeta\colon\state\to\Re$.  For the MDP with $d$ states, each function is viewed as a vector in $\Re^d$ with entries $\{h_\zeta^*(x^i) : 1\le i\le d\}$.  A vector field $\clV\colon\Re^d\to\Re^d$ is constructed so that these functions solve the ODE 
\[
\ddzeta h^*_\zeta  =  \clV(h^*_\zeta)\,,
\qquad\text{
with boundary condition $h^*_0\equiv 0$.
 }
 \]

One step in the construction of $\clV$ is differentiating each side of the dynamic programming equations; a starting point of the 50 year old sensitivity theory of  \cite{sch68}, and  more recent \cite{sutmcasinman99}.
  More closely related is the sensitivity theory surrounding Perron-Frobenius eigenvectors that appears in the theory of large deviations \cite[Prop.~4.9]{konmey03a}.
 The goals of this prior work are different, and we are not aware of comparable algorithms that simultaneously solve the family of control problems.

  The optimal control formulation is far more general than in the aforementioned work \cite{tod07,guaragwil14,meybarbusyueehr15}, as it allows for inclusion of exogenous randomness in the MDP model.   The dynamic programming equations become significantly  more complex in this generality, so that in particular, the Perron-Frobenious computational approach used in  prior work is no longer applicable.


In addition to its value as a computational tool,  there is a significant benefit to solve the entire collection of optimal control problems for a range of the parameter $\zeta$.  For example, this provides a means to understand the tradeoff between state cost and control effort.     Simultaneous computation of the optimal policies is also an essential ingredient of the distributed control architecture introduced in  \cite{meybarbusyueehr15}.   

The ODE algorithm
is easily implemented for  problems of moderate size.   In this paper an example is provided in which the the size of the state space $d$ is greater than 1,000;   the action space is an open subset $\Re^{d-1}$ since actions correspond to randomized decision rules.   The optimal solutions for the desired range of $\zeta$ were obtained in less than one hour using a standard laptop running Matlab.


The remainder of the paper is organized as follows.   
\Section{s:design} describes the new Kullback--Leibler cost criterion and numerical techniques for the MDP solutions.    This is applied to a path-finding problem in \Section{s:drone}.
Conclusions and topics for future research are contained in \Section{s:conc}.

\section{MDPs with Kullback--Leibler Cost} 
\label{s:design}

\subsection{MDP model}   

The dynamics of the MDP are assumed of the form \eqref{e:TodMDPu},  where the action space consists of a convex subset of probability mass functions (pmf) on $\state$.  An explanation of the one-step reward  \eqref{e:oneStepReward} will be provided after a few preliminaries.

A transition matrix $P_0$ is given that describes  nominal (control-free) behavior.   It is assumed to be \textit{irreducible and aperiodic}.   It follows that $P_0$ admits a unique invariant pmf, denoted $\pi_0$.  For any other transition matrix, with unique invariant pmf $\pi$, the
 \textit{Donsker-Varadhan rate function} is denoted,
\begin{equation}
K(P\| P_0) = \sum_{x,x'} \pi(x) P(x,x')   \log \Bigl(\frac{P(x,x') }{P_0(x,x')} \Bigr)
\label{e:DVrate}
\end{equation}  
under the usual convention that ``$0\log(0) = 0$''.  
It is called a ``rate function'' because it defines the relative entropy rate between two stationary Markov chains, see
\cite{demzei98a}.


As in \cite{tod07,guaragwil14,meybarbusyueehr15}, the rate function is used here 
to model the cost of deviation from the nominal transition matrix $P_0$. 
The two control objectives surveyed in the introduction will be specialized as follows, based on the utility function    $\util\colon\state\to\Re$ and a scaling parameter $\zeta\ge 0$.
For  the finite-horizon optimal control problem,  

\begin{equation}
\clW^*_T(x,\zeta) = \max \sum_{t=0}^{T} \Expect_x[ \zeta\util(X(t))   - \cKL(X(t),P(t)) ] \,, 
\label{e:TCobjectiveReward}
\end{equation}
where   
the expectation is conditional on $X(0)=x$, and 
\begin{equation}
\cKL(x, P) = D(P(x,\varble)\| P_0(x,\varble))\eqdef
\sum_{x'}  P(x,x') \log \Bigl(\frac{P(x,x') }{P_0(x,x')} \Bigr)
\label{e:cKL}
\end{equation}
for any $x\in\state$  and transition matrix $P$.

The average reward optimization problem is analogous:
\begin{equation} 
\eta^*(\zeta) =\max\Bigl( \liminf_{T\to\infty}\frac{1}{T} \sum_{t=0}^{T-1} \Expect_x \left [ \zeta \util(X(t))   - \cKL(X(t),P(t)) \right ]   \Bigr)\,.
\label{e:ARobjectiveReward_def}
\end{equation}
In each case, the  maximum is over all transition matrices $\{P(t)\}$. 
The average reward optimization problem can be cast as
the solution to the convex optimization problem,
\begin{equation}
\eta^*_\zeta =\max_{\pi,P} \bigl \{\zeta \pi(\util) -  K(P\| P_0)  :  \pi   P  = \pi \bigr\}
\label{e:ARobjectiveReward}
\end{equation}
where  the maximum is over all transition matrices.    

In this context, the one-step reward   appearing in (\ref{e:TCobjectiveReward_intro}, \ref{e:ARobjectiveReward_intro}) is a function of pairs $(x,P)$:
\begin{equation}  
 \reward(x,P)  \eqdef \zeta \util(x) -   \cKL(x, P)
\label{e:cKLa}
\end{equation}
for any $x\in\state$  and transition matrix $P$.
There is practical value to considering a parameterized family of reward functions.  For one, it is useful to understand the sensitivity of the control solution to the relative weight given to utility    and the penalty on control action.   This is well understood in classical linear control theory -- consider for example the celebrated symmetric root locus in linear optimal control \cite{fraworpow97}.


\paragraph*{Nature \&\ nurture}   

Exogenous randomness from nature imposes additional constraints in the optimal control problem  \eqref{e:TCobjectiveReward} or \eqref{e:ARobjectiveReward_def}.   

It is assumed that the state space is the cartesian product of two finite sets:  $\state= \stateu\times\staten$, and the state is similarly expressed $X(t)=(X_u(t), X_n(t))$.   At a given time $t$ it is assumed that $X_n(t+1)$ is conditionally independent of the input at time $t$,  given the value of $X(t)$.   This is formalized by the  following conditional-independence assumption: 
\begin{equation}
P(x,x') = R(x, x_u') Q_0(x,x_n') ,\quad x=(x_u,x_n) \in\state, \ x_u'\in\stateu,\ x_n'\in\staten 
\label{e:PQ0R}
\end{equation} 
The matrix $R$ defines the randomized decision rule for $X_u(t+1)$ given $X(t)$.  
The matrix $Q_0$ is fixed and models the distribution of   $X_n(t+1)$ given $X(t)=x$,
 and each are subject to the pmf constraint:
$
\sum_{x_u'} R(x, x_u')=\sum_{x_n'}  Q_0(x,x_n') =1
$
for each $x$. 

Subject to the constraint \eqref{e:PQ0R}, the two optimal control problems (\ref{e:ARobjectiveReward_def}, \ref{e:cKLa}) are transformed to  the final forms considered in this paper:
\begin{eqnarray}
\clW^*_T(x,\zeta) &=& \max \sum_{t=0}^T  \Expect_x[  \reward(X(t),R(t))   ]
\label{e:TCobjective_NN}
\\[.2em]
 \eta^*(\zeta) &=&\max\Bigl\{ \liminf_{T\to\infty}   \frac{1}{T}  \sum_{t=0}^{T-1}   \Expect_x[  \reward(X(t),R(t)) ]  \Bigr\}
 \label{e:ACobjective_NN}
\end{eqnarray}
where in each case the maximum is over sequences of randomized decision rules $\{R(0),\dots,R(T)\}$,     
\begin{equation} 
\begin{aligned}
 \reward(x,R)  &\eqdef \zeta \util(x) -  \cKL(x,R) 
 \\[.2em]
 \text{\it and}\quad 
\cKL(x,R)& \eqdef \sum_{x'}  P(x,x') \log \Bigl(\frac{P(x,x') }{P_0(x,x')} \Bigr)  
 =   \sum_{x'_u}  R(x,x'_u) \log \Bigl(\frac{R(x,x'_u) }{R_0(x,x_u')} \Bigr)   
\end{aligned}
\label{e:cKL}
\end{equation}

\subsection{Notation}

For any transition matrix $P$, an invariant pmf  is interpreted as a row vector, so that invariance can be expressed $\pi P=\pi$.   Any function $f\colon\state\to\Re$ is interpreted as a $d$-dimensional column vector,  and we use the standard notation $Pf\, (x) =   \sum_{x'}P(x,x')f(x')$,  $x\in\state$.
The \textit{fundamental matrix} is   the inverse, 
\begin{equation}
Z = [I - P + 1\otimes \pi]^{-1}
\label{e:fundKernGen}
\end{equation}
where $1\otimes \pi$ is a matrix in which each row is identical, and equal to $ \pi$.  
If $P$ is irreducible and aperiodic, then it can be expressed as the power series  
$Z = \sum_{n=0}^\infty  [P - 1\otimes \pi]^n
$,
with $ [P - 1\otimes \pi]^0 \eqdef I$  (the $d\times d$ identity matrix),
and $ [P - 1\otimes \pi]^n  = P^n - 1\otimes \pi$  for $n\ge 1$.  

Any function $g\colon\state\times\state\to\Re$ is regarded as an unnormalized log-likelihood ratio:  Denote for $x,x'\in\state$,
\begin{equation}
P_g(x,x') \eqdef P_0(x,x')\exp\bigl(  g( x' \mid x )  -  \EFn{g}(x)    \bigr),  
\label{e:Ph-a}
\end{equation}
in which $g( x' \mid x )$ is the value of $g$ at $(x,x')\in\state\times\state$, and $ \EFn{g}(x)$ is the normalization constant,
\begin{equation}
    \EFn{g}(x)     
\eqdef  \log\Bigl( \sum_{x'} P_0(x,x')\exp\bigl(  g(x'\mid  x)     \bigr) \Bigr)
\label{e:EFn}
\end{equation}
The rate function can be expressed in terms of  its  invariant pmf $\pi_g$, the bivariate pmf  $\Pi_g(x,x') = \pi_g(x) P_g(x,x')  $, and the log moment generating function \eqref{e:EFn}:
\begin{equation}
\begin{aligned}
K(P_g\| P_0) &= \sum_{x,x'} \Pi_g(x,x')    \bigl[  g( x' \mid x )  -  \EFn{g}(x) \bigr]
\\
&= \sum_{x,x'}\Pi_g(x,x')      g( x' \mid x )  -     \sum_x \pi_g(x) \EFn{g}(x)  
\end{aligned}
\label{e:DVrate-h}
\end{equation}  

The unusual notation is introduced because  $ g( x' \mid x )  $ will take the form of a conditional expectation in all of the results that follow:  given any function $h\colon\state\to \Re$ we denote
\begin{equation}
h(x'_u\mid  x)  =  \sum_{x_n'} Q_0(x,x_n') h(x_u',x_n') \, .
\label{e:hmid}
\end{equation}
In this case the transformation only transforms the dynamics of $\bfmX_u$:
\[
P_h(x,x')  = R_h(x,x_u') Q_0(x,x_n') \,,\quad
R_h(x,x'_u) \eqdef R_0(x,x'_u)\exp\bigl( h( x'_u \mid x )  -  \EFn{g}(x)    \bigr)\, .
\]

\subsection{ODE for finite time horizon}
\label{s:ODETC}

Here an ODE is constructed to compute the   value functions $\{\clW^*_\tau (x,\zeta)  : 1\le \tau \le T\,, \zeta\ge 0\}$.    To aide exposition it is helpful to first look at the general problem:   Assume that the state space $\state$ is finite, the action space $\U$ is \textit{general},  and let $\{P_u(x,x')\}$ denote the controlled transition matrix.   The  one-step reward on state-action pairs is of the form $\reward(x,u) = \zeta\util(x) - c(x,u)$,  where $c\colon\state\times\U\to\Re_+$.   Assume that  $c(x,u)\equiv 0$  for a unique value $u=u_0$.

For each $1\le \tau \le T$ denote, as in \eqref{e:TCobjectiveReward_intro},
\begin{equation}
\clW^*_\tau(x,\zeta) = \max \sum_{t=0}^\tau  \Expect_x[ w(X(t), U(t)) ]     
\label{e:TCobjectiveReward_gen}
\end{equation}  
where the maximum is over all admissible  
inputs $\{U(t)=\phi_t(X(0),\dots,X(t))\}$.
Each value function can be regarded as the maximum over functions   $\{\phi_t\}$  (subject to measurability conditions and 
hard constraints on the input).  It is assumed that the maximum \eqref{e:TCobjectiveReward_gen} is finite for each $(x ,\zeta)$. 
 
The dynamic programming equation (principle of optimality) holds:  for $\tau\ge 1$,
\begin{equation}
\clW^*_\tau(x,\zeta) = \max_u \Bigl\{  \zeta\util(x)  -  c(x,u)     + \sum_{x'} P_u(x,x') \clW^*_{\tau-1} (x')  \Bigr\}
\label{e:TCobjectiveReward_gen_PoO}
\end{equation}  
Assume that a maximizer  $\phi^*_{\tau-1,\zeta}(x)$ exits for each $\tau$,$\zeta$, and $x$.

A crucial observation is that for each $x$, the value function appearing in \eqref{e:TCobjectiveReward_gen} is the maximum
of functions that are affine in $\zeta$.  It follows that $\clW^*_\tau(x,\zeta) $ is convex as a function of $\zeta$, and hence absolutely continuous.  Consequently, the right derivative   $H^*_\tau(x,\zeta) \eqdef \ddzetap  \clW^*_\tau(x,\zeta) $ exists everywhere.  A recursive equation follows from 
\eqref{e:TCobjectiveReward_gen_PoO}:
\begin{equation}
H^*_\tau(x,\zeta) = \util(x)       + \sum_{x'} \cP_{\tau-1,\zeta}(x,x') H^*_{\tau-1} (x',\zeta)  
\label{e:TCobjectiveReward_ODEa}
\end{equation}  
where $\cP_{\tau-1,\zeta}(x,x') = P_{u^*}(x,x')$ with $u^* =  \phi^*_{\tau-1,\zeta}(x)$.   

In matrix notation this becomes $H^*_\tau =  \cZ_{\tau-1,\zeta} \util$, where
 $\cZ_{0,\zeta} =I$, and
 for any $1\le \tau\le T$,
\begin{equation}
\cZ_{\tau-1,\zeta} = I +  \cP_{\tau-1,\zeta}+\cP_{\tau-1,\zeta}\cP_{\tau-2,\zeta}
			+ \cP_{\tau-1,\zeta}\cP_{\tau-2,\zeta} \cdots \cP_{0,\zeta}
\label{e:fund_matrix_finite}
\end{equation}
 This is similar to a truncation of the  power series representation of the fundamental matrix \eqref{e:fundKernGen}.

Denote $\clW^*_\zeta(x) =\{ \clW^*_k(x,\zeta)  : 0\le k\le T\}$, regarded as a vector in $\Re^{|\state|\times(T+1)}$, parameterized by the non-negative constant $\zeta$.  The following result follows from the preceding arguments: 
\begin{theorem}
\label{t:TCODE}  
The family of functions $\{\clW^*_\zeta\}$ solves the ODE
$\displaystyle 
\ddzetap \clW^*_\zeta =\clV( \clW^*_\zeta )$,  $ \zeta\ge 0$,  
with boundary condition $\clW^*_0 = 0$.   The vector field can be described in block-form as follows, with $T+1$ blocks:
\[
\ddzetap \clW^*_k(\varble,\zeta)  =\clV_k( \clW^*_\zeta )\,,\quad 0\le k\le T\, .
\]
The identity $\clV_0(\clW) = \util$ holds for any $\clW$.
For $k\ge 1$, the right hand side depends on  its argument only through the associated policy:  for any sequence of functions $\clW =(\clW_0,\dots,\clW_T)$,
\[
\begin{aligned}
\clV_k(\clW)  &=  Z_{k-1} \util 
\\
\text{where} \qquad
  Z_{k-1} &= 
 I +  P_{k-1}+P_{k-1}P_{k-2}
			+ P_{k-1}P_{k-2}  \cdots P_0 	 
			\\
			P_i(x,x') &= 	P_{\phi_{i}(x)} (x,x')\,, \quad \text{all $i$, $x$, $x'$,} 
\\
\phi_i(x)  &=  \argmax_u   \Bigl\{  -  c(x,u)     + \sum_{x'} P_u(x,x') \clW_i (x')  \Bigr\}   \,, \qquad 1\le i, k\le T.
\end{aligned}
\]\hfill \sq
\end{theorem}

The theorem provides valuable computational tools for models of moderate 
cardinality and moderate time-horizon.  
Two questions remain:
\begin{romannumx}
\item   What is $\phi_i$ for the problem under study in this paper?
\item  Can a tractable ODE be constructed in infinite-horizon optimal control problems?
\end{romannumx}
The answer to the second question is the focus of Section \ref{s:AROE}.  
The answer to (i) is contained in the following.   For any function $\clW \colon\state\to\Re$, denote
\[
R_\clW(x,\varble) = \argmax_R  \Bigl\{ \reward(x,R) +\sum_{x'} P(x,x') \clW(x') \Bigr\}\,,\quad x\in\state\, ,
\]
subject to the constraint that $P$ depends on $R$ via \eqref{e:PQ0R}, and with $\reward$ defined in \eqref{e:cKL}.

\begin{proposition}
\label{t:phiR}  
For any function $\clW $ the maximizer $R_\clW$ is unique and can be expressed
\[
R_\clW(x,x_u') = R_0(x,x_u')  \exp\bigl(  \clW(x'_u\mid  x)  -   \Lambda(x)    \bigr) 
\]
where  $\clW(x'_u\mid  x)  =  \sum_{x_n'} Q_0(x,x_n') \clW(x_u',x_n')$  for each $ x\in\state$, $x_u'\in\stateu$,
and $\Lambda(x)$ is a normalizing constant, defined so that $R_\clW(x,\varble)$ is a pmf for each $x$.
\end{proposition}

%

\begin{proof}
 Given the form of the reward $\reward$ and the constraint on $P$, 
the optimization problem of interest here can be written, for each $x$, as
 \[
R_\clW (x,\varble) = \argmax_\mu  \bigl\{   \mu(\haclW  ) - D(\mu\| \mu_0)  \bigr\}
\]
where the variable $\mu(\varble) $ represents $R(x,\varble)$,   $\mu_0=R_0(x,\varble)$, and
\[
\mu(\haclW  ) =  \sum_{x' = (x'_u,x'_n)} R(x,x'_u)Q_0(x,x'_n) \clW(x'_u,x'_n) = \sum_{x'_u} \mu(x'_u) \clW(x'_u\mid  x) 
\]
The proposition is a consequence of this combined with   Theorem 3.1.2  of  \cite{demzei98a}  (i.e., convex duality between relative entropy and the log moment generating function).
\end{proof}

It follows from the proposition that the vector field is smooth in a neighborhood of the optimal solution $\{\clW^*_\zeta : \zeta\ge 0\}$.   These results are central to the average-reward case considered next.

\subsection{Average reward formulation}
\label{s:AROE}

We consider now the case of average reward \eqref{e:ACobjective_NN}, subject to the structural constraint \eqref{e:PQ0R}.    The associated average reward optimization equation (AROE) is expressed as follows: 
\begin{equation}
\max_R\Bigl\{
\reward(x,R)
+\sum_{x'} P(x,x') h^*_\zeta(x') \Bigr\} = h^*_\zeta(x) + \eta^*(\zeta)
\label{e:AROE}
\end{equation}
 In which   $\eta^*(\zeta)$ is the optimal average reward, and $h^*_\zeta$ is the \textit{relative value function}.  The maximizer defines a transition matrix:
\begin{equation}
\cP_\zeta =\argmax_P \bigl \{\zeta \pi(\util) -  K(P\| P_0)  : \pi   P  = \pi  \bigr\}
\label{e:ARobjective}
\end{equation}
Recall that the relative value function is not unique, since  a new solution is obtained by adding a non-zero constant;  the normalization  $h^*_\zeta(\xz)=0$ is imposed, where $\xz\in\state$ is a fixed state.

   The proof of \Theorem{t:IPD}~(i) is a consequence of \Prop{t:phiR}.  The second result is obtained on combining   Lemmas~B.2--B.4 of   \cite{busmey18a}.
\begin{theorem}
\label{t:IPD}
There exist optimizers $\{\cpi_\zeta, \cP_\zeta : \zeta\in\Re\}$,  and solutions to
the AROE    $\{h^*_\zeta,\eta^*(\zeta): \zeta\in\Re\}$ with the following properties: 
\begin{romannumx}

\item  
The optimizer $\cP_\zeta$ can be obtained from the relative value function $h^*_\zeta$ as follows:  
\begin{equation}
\cP_\zeta(x,x')
\eqdef P_0(x,x')\exp\bigl(  h_\zeta(x'_u\mid  x)  -  \EFn{h_\zeta}(x)    \bigr)
\label{e:Pzeta-a}
\end{equation}
where  for $ 
x\in\state$, $x_u'\in\stateu$,
\begin{equation}
 h_\zeta(x'_u\mid  x)  =  \sum_{x_n'} Q_0(x,x_n') h^*_\zeta(x_u',x_n'),  
\label{e:hmidzeta}
\end{equation}
and $\EFn{h_\zeta}(x) $ is the normalizing constant \eqref{e:EFn} with   $h=h_\zeta$.

\item
$\{\cpi_\zeta, \cP_\zeta, h^*_\zeta,\eta^*(\zeta): \zeta\in\Re\}$ are continuously differentiable in the parameter $\zeta$.
\qed
\end{romannumx}
\end{theorem}

Representations for the derivatives in \Theorem{t:IPD}~(ii),  in particular the derivative of $\EFn{h_\zeta^*}$ with respect to $\zeta$,  lead to a representation for the ODE   used to compute the    transition matrices $\{\cP_\zeta\}$.

It is convenient to generalize the problem slightly here:  let $\{h_\zeta^\circ : \zeta\in\Re\}$ denote a family of functions on $\state$, continuously differentiable in the parameter $\zeta$.  They are not necessarily relative value functions, but we maintain the structure established in  \Theorem{t:IPD} for the  family of transition matrices.  Denote,
\begin{equation}
h_\zeta (x'_u\mid  x)   =  \sum_{x_n'} Q_0(x,x_n') h^\circ_\zeta(x_u',x_n'),  
\quad 
x\in\state,\ x_u'\in\stateu
\label{e:hmidcirc}
\end{equation}
and then define as in \eqref{e:Ph-a},
\begin{equation}
 P_\zeta(x,x')
\eqdef P_0(x,x')\exp\bigl(  h_\zeta(x'_u\mid  x )  -  \EFn{h_\zeta}(x)    \bigr)
\label{e:Pzeta}
\end{equation} 
The function   $\EFn{h_\zeta}\colon\state\to\Re$ is a normalizing constant, exactly as in \eqref{e:EFn}:
\[
    \EFn{h_\zeta^\circ}(x)     
\eqdef  \log\Bigl( \sum_{x'} P_0(x,x')\exp\bigl(  h_\zeta(x'_u\mid  x)     \bigr) \Bigr)
\]


We begin with a general method to construct  a family of functions $\{h_\zeta^\circ : \zeta\in\Re\}$ based on an ODE.     The ODE is expressed,  
\begin{equation}
\ddzeta h_\zeta^\circ  =  \clV(h_\zeta^\circ)\,,\qquad \zeta\in\Re,
\label{e:hODE}
\end{equation}
with boundary condition $h_0^\circ\equiv 0$.  A particular instance of the method will result in $h_\zeta^\circ = h_\zeta^*$ for each $\zeta$.
 Assumed given is a mapping $\preclH$ from transition matrices to functions on $\state$.   Following this, the vector field $\clV$ is obtained through the following two steps:  For a function $h\colon\state\to\Re$,  
\begin{romannumx}
\item Define a new transition matrix via \eqref{e:Ph-a},
\begin{equation}
P_h(x,x') \eqdef P_0(x,x')\exp\bigl(  h(x_u' \mid x)  -  \EFn{h}(x)    \bigr),\quad x,x'\in\state,
\label{e:Ph}
\end{equation}
in which $  h(x_u' \mid x) = \sum_{x_n'} Q_0(x,x_n') h(x_u',x_n')$,  and $ \EFn{h}(x)  $ is a normalizing constant.

\item  Compute $\preH = \preclH(P_h)$, and define $\clV(h) = \preH$.
It is assumed   that the functional $\preclH$ is constructed so that  $ \preH(\xz)=0$ for any $h$.

\end{romannumx}


 
We now specify the functional  $\preclH$, whose domain consists of  transition matrices that are irreducible and aperiodic.    For  any transition matrix  $P$ in this domain, the fundamental matrix $Z$ is obtained using \eqref{e:fundKernGen}, and then $\preH=\preclH(P)$ is defined as
\begin{equation}
\preH(x) = \sum_{x'} [ Z(x,x')-Z(\xz,x') ] \util (x'),\qquad x\in\state
\label{e:fishP}
\end{equation}
The function $\preH$   is a solution to Poisson's equation,
\begin{equation}
P \preH =\preH -\util +\meanutil\,, \qquad \text{where $  \meanutil  \eqdef  \pi(\util)  \eqdef \sum_x\pi(x) \util(x)$.
}
\label{e:fish}
\end{equation}

\begin{theorem}
\label{t:IDPODE}
Consider the ODE 
\eqref{e:hODE}
with boundary condition $h_0^\circ\equiv 0$,
and with  $\preH=\preclH(P)$ defined using \eqref{e:fishP} for each transition matrix $P$
that is irreducible and aperiodic.
The solution to this ODE exists, and the resulting functions $\{ h^\circ_\zeta : \zeta\in\Re\}$    
coincide with the relative value functions $\{h^*_\zeta: \zeta\in\Re\}$.  Consequently,   $\cP_\zeta = P_{h_\zeta}$ for each $\zeta$.
\end{theorem}

\begin{proof}
The proof requires validation of the representation $  H^*_\zeta = \preclH(\cP_\zeta)$ for each $\zeta$, where $h^*_\zeta$ is the relative value function, 
$\cP_\zeta$ is defined in \eqref{e:ARobjective}, and
\begin{equation}
  H^*_\zeta =\ddzeta h^*_\zeta 
\label{e:Hrep}
\end{equation} 
Substituting the maximizer $\cP_\zeta$ in the form \eqref{e:Pzeta-a} into the AROE gives the fixed point equation $ \zeta \util   + \EFn{h^*_\zeta} = h^*_\zeta + \eta^*(\zeta) $.
%
Differentiating each side   then gives,
\begin{equation}
   \util   +  \cP_\zeta H^*_\zeta= H^*_\zeta + \ddzeta \eta^*(\zeta) .
\label{e:fisheta}
\end{equation}
This is Poisson's equation, and it follows that $\cpi_\zeta(\util) =  \ddzeta \eta^*(\zeta) $.  Moreover,  since $h^*_\zeta(\xz)=0$ for every $\zeta$,  we must have $H^*_\zeta(\xz)=0$ as well.  Since the solution to Poisson's equation with this normalization is unique, we conclude that  \eqref{e:Hrep} holds, and hence $H^*_\zeta = \preclH(\cP_\zeta)$ as claimed.  
\end{proof}

\def\Ltarget{l^{\bullet}}

\def\stateL{{\sf X}^L}
\def\wind{\omega}

\section{Example} 
\label{s:drone}
 
We consider a variant of the example of  \cite{alsgonsmi13} in which a UAV 
(unmanned aerial vehicle) 
 needs to reach a target subject to energy costs, and subject to disturbances from wind.   The location of the UAV at time $t$ is denoted $L_t$, and evolves according to the controlled linear dynamics:
\begin{equation}
L_{t+1} = L_t + W_t + U_t  
\label{e:modelLnat}
\end{equation}
where $\bfmU=\{U_t\}$ is the control sequence, and  $\bfmW=\{W_t\}$  models the impact of the wind. There are $d_L$ locations across a two-dimensional grid. 

Wind is location-dependent: It is assumed that the  wind profile over the region is determined by a stochastic process $\bfmN =\{ N_t\} $ and a function $\wind$ such that for each $t$,
\[
W_t=\wind(L_t, N_t ).
\]
The process $\bfmN$  is assumed to be Markovian with finite state space $  \{1,\dots,d_N\}$, and state transition matrix denoted $Q_0$.  This is the
 nature component of the MDP model,  with state process $X_t=(L_t,N_t)$,  $t\ge 0$.

A nominal model is described by a randomized policy in which $U_t=0$ with high probability.   The specific form used in the experiments 
was constructed as follows.   On denoting  $  L^+_t = L_t + W_t$, a transition matrix $R_0^L$ is constructed with the interpretation
\[
R_0^L(l^+, l')  = \Prob\{ L_{t+1} =  l' \mid L^+_t  = l^+ \}\,,\quad t\ge 0\, .
\]
The nominal randomized strategy is the $d_L\times d_L$ matrix,
\[
R_0(x, u)  = \Prob\{ U_t = u\mid X_t  = x \} = R_0^L(l+\wind(l,n) , l+\wind(l,n)+u ) ,   \qquad x=(l,n)\, .
\]
The overall transition matrix is the product:
\[
P_0(x,x') = R_0^L(l+w(n,l), l') Q_0(n,n'),\qquad x=(n,l),\ x'=(n',l')\, .
\]

 \begin{figure}[h]
\Ebox{.95}{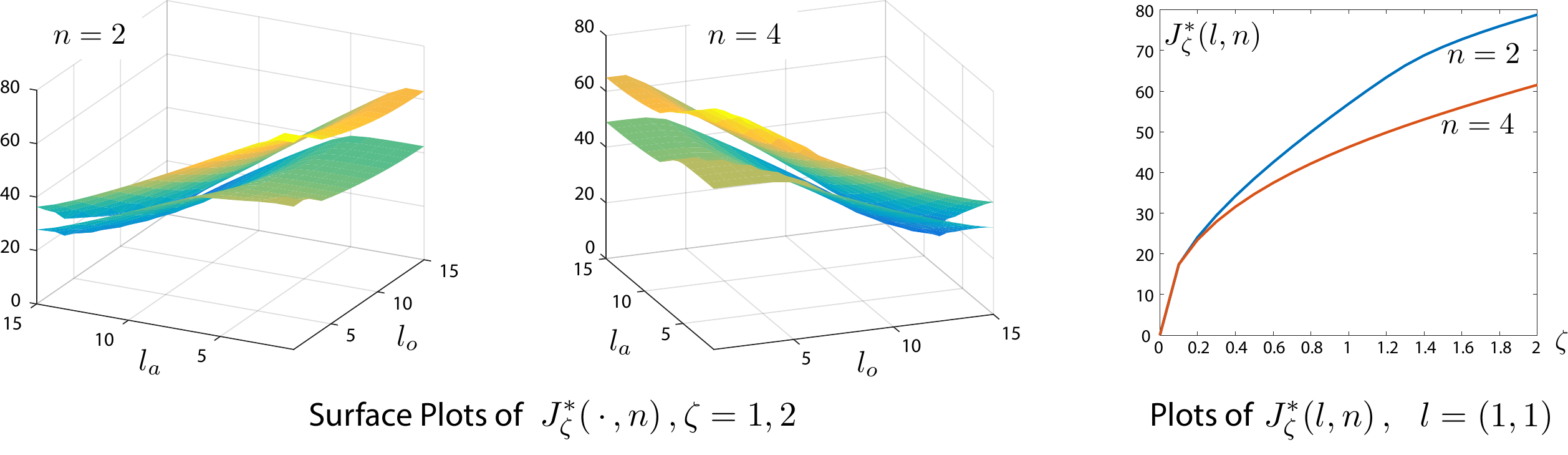} 
\vspace{-3.5ex}
\caption{\small Cost to go for two values of the initial value $n=N_0$,  $n =2,4$.   Each surface plot indicates values of $J^*_\zeta(\varble,n)$ for $\zeta=1$ and $\zeta=2$.  The one of larger magnitude corresponds to $\zeta=2$.    The plot at the right shows $J^*_\zeta(l,n)$ as a function of $\zeta$ for these values of $n$,  and $l=(1,1)$. }
\label{f:CostToGo15x15n2n4}
\vspace{-1.5ex}
\end{figure}

The goal of the control problem is to reach a target location $\Ltarget$ and remain there.  To ensure that the set $\{ (\Ltarget,n) :  1\le n\le d_N\}$ is absorbing, a separate rule is imposed on $R_0$ for these states:   $W_t + U_t  =0$ if $L_t = \Ltarget$.

The reward function $\util$ is taken to be a scaled negative cost:  $ \util = -  c$,  where $c\colon\stateL\to\Re_+$,
with $c(\Ltarget)=0$ and $c(l)>0$ for $l\neq \Ltarget$.    The optimal steady-state mean is zero in this model, and the relative value function is the negative of the cost to go:
\begin{equation}
-h^*(x) = 
J^*(x) \eqdef \min\Expect_x\Bigl[\sum_{t=0}^{\tau_\bullet} 
	\bigl\{\zeta c(L_t) +\cKL(X_t,R(t) ) \bigr \} \Bigr]
\label{e:RoverCTG}
\end{equation}
where $\tau_\bullet$  (unknown a-priori) is the first hitting time to $\Ltarget$.   An example is illustrated in \Fig{f:CostToGo15x15n2n4} --- the details are provided in the following.

 \begin{figure}[h]
\Ebox{0.85}{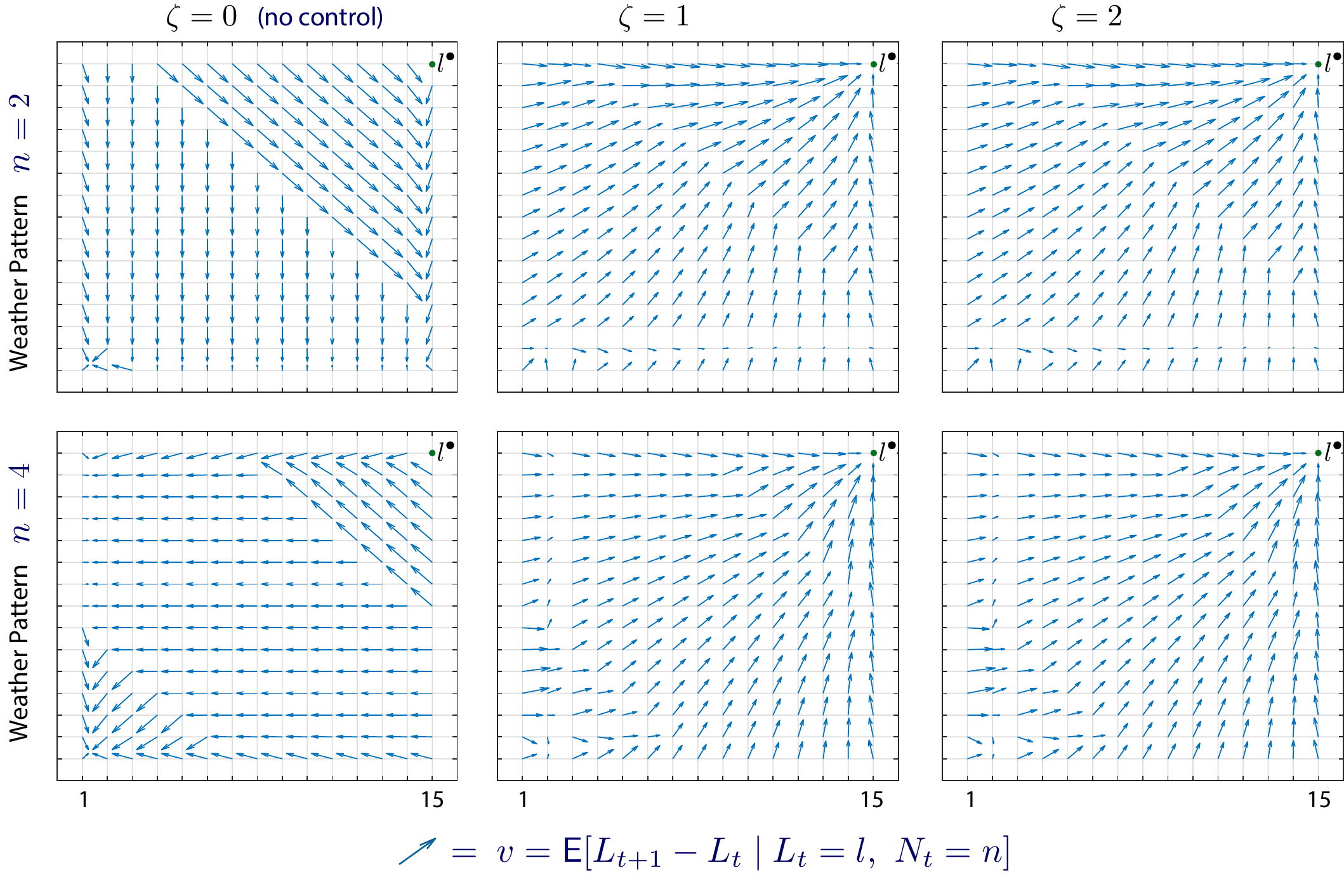} 
\vspace{-3ex}
\caption{\small Vector field $v(l,n)$ for two values of $n$, and $\zeta=0,1,2$:  see eqn.~\eqref{e:velocity}}
\label{f:Quiver15by15n2n4.pdf}
\vspace{-.25ex}
\end{figure}

\paragraph{Details of the numerical experiment}

The  set of locations $\stateL$ is taken to be a rectangular grid of the form $\stateL = \{ (i,j) : 1\le i\le d_a,\  1\le j \le d_o\}$, in which   $d_a, d_o\ge 2$ and $d_L=d_a\times d_o$  (the subscripts are  meant to suggest latitude and longitude).  The function $c$ appearing in \eqref{e:RoverCTG} was taken to be the indicator function,  $c(l)=\ind\{l\neq \Ltarget\}$,  with $\Ltarget=(d_a,d_o)$.

The values  $d_a=d_o=15$, and $d_N=5$ are fixed throughout.
The size of the state space is thus $d_a\times d_o\times d_N= 1,125$,  and the action space is a subset of the simplex in $\Re^{1,125}$.

The transition matrix for nominal control was taken   of the following form:
\[
R_0^L(l, l')   = \kappa(l)   \exp\Bigl\{-\frac{1}{2\sigma^2_u} \| l'-l\|^2 \Bigr\}  ,\quad l,l'\in\stateL\,,
\]
where $\kappa(l) >0$ is chosen so that $R_0^L(l, \varble)$ is a pmf on $\stateL$ for each $l\in\stateL$.   The value $\sigma^2_u = 1/2$ was used in the numerical results that follow.

The Markov chain $\bfmN$ was taken to be a skip-free symmetric random walk on the integers $\{ 1,\dots,d_N\}$.   For a fixed $\delta_n\in(0,1)$ the probability of transition is $Q_0(n,n+ 1) =Q_0(n,n-1) =\half \delta_n $,
where addition is modulo $d_N$,  
and $Q_0(n,n) =1-\delta_n$ for any $n$.
Recall that this means
\[ 
\Prob\{N_{t+1} = n+1 \mid N_t = n\}
=
\Prob\{N_{t+1} = n-1 \mid N_t = n\} =  \half \delta_n. 
\]
The value $\delta_n=0.05$ was chosen in these experiments.

Recall that $\wind\colon \stateL\to \ZZ^2$ is used to defined the wind process $\bfmW$.  
For each value of $n$,  the function $\wind(\varble,n)$ can be interpreted as a vector field on $\stateL$.   For each $n$,  a slowly varying continuous function was constructed on the two-dimensional rectangle $[1,d_a]\times[1,d_o]$. The function $\wind(\varble,n)$ was  taken to be its quantization to the lattice $\stateL$.
 The values were restricted to the set of pairs $\{(i,j): |i|\le 1,\ |j|\le 1\}$.

The family of optimal policies was obtained using the ODE method,   and the solution for three values of $\zeta $ is illustrated in  \Fig{f:Quiver15by15n2n4.pdf}.    Each of the arrows shown is proportional to  the conditional expectation:
\begin{equation}
v(l,n) \eqdef 
\Expect[L_{t+1} - L_t \mid L_t = l\, \ N_t=n]
\label{e:velocity}
\end{equation}
in which $l\in\stateL$ is the position on the grid.  The figure shows only the values $n=2$ and $n=4$ (the most interesting to view because of obvious spatial variability). 


 \begin{figure}[h]
\vspace{-2ex}
\Ebox{0.4}{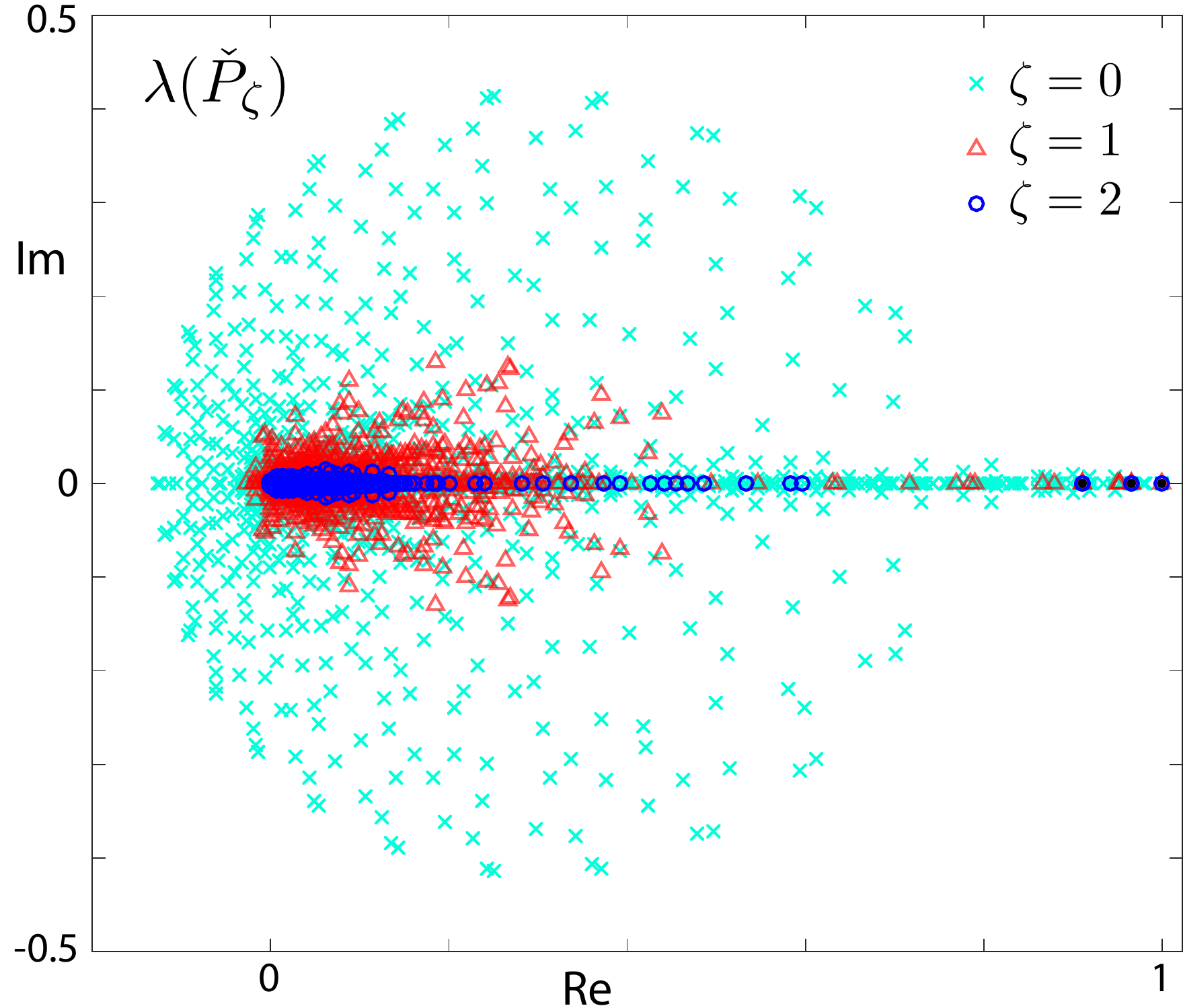} 
\vspace{-2ex}
\caption{\small Eigenvalues of $\cP_\zeta$ }
\label{f:EigenvalueComparison15x15}
\vspace{.25ex}
\end{figure}

 
If the position $l=(l_a,l_o)$  is far from the boundary of $\stateL$, say,  $\min(l_a,l_o)\ge 4$ and  $\min(d_a-l_a,d_o-l_o)\ge 4$, then
\[
\text{  $\Expect[U_t \mid L_t = l\, \ N_t=n]\approx 0$ and 
$
v(l,n) \approx \wind(l,n)$,  \quad $\zeta=0$}
\]

For the case $\zeta=1$ the vector field is transformed so that vectors near the target state point in this direction; for $\zeta=2$ this behavior is  more apparent.   For states far from the target the control effort seems to be lower -- most likely the optimal policy waits for more favorable weather that will push the UAV in the North-East direction.


The eigenvalues of $\cP_\zeta$ are shown in \Fig{f:EigenvalueComparison15x15}
for $\zeta=0,1,2$.    Most of the eigenvalues are driven near zero for $\zeta=2$. 
Those three that are independent of $\zeta$ are  the three eigenvalues of $Q_0$,
$\{0.9095, 
    0.9655, 1   \}$.

While the vector field and eigenvalues change significantly when $\zeta$ is doubled from $1$ to $2$, the cost to go  $J^*$ defined in \eqref{e:RoverCTG} grows relatively slowly with $\zeta$.  Shown on the right hand side of \Fig{f:CostToGo15x15n2n4} are comparisons for these two values of $\zeta$.  One plot with $n=2$ and the other $n=4$.
The plot on the far right shows $J_\zeta^*(l,\zeta)$ for $0\le\zeta\le 1$ and $l=(1,1)$  (the location farthest from $\Ltarget$).   

\textit{These plots are easily obtained because of the nature of the algorithm:}   the optimal policy and value function are generated for any range of $\zeta$ of interest.

%

\section{Conclusions} 
\label{s:conc}

The ODE approach for solving MDPs has  simple structure for the class of models considered in this paper.  We are currently looking at approaches to approximate dynamic programming as has been successful in the unconstrained model  \cite{tod09}.

It is likely that the ODE has special structure for other classes of MDPs, such as the  ``rational inattention'' framework of \cite{sim06,sharagmey16}.  The computational efficiency of this approach  will depend in part on numerical properties of the ODE, such as its sensitivity for complex models.   Applications to distributed control were the original motivation for this work, with particular attention to ``demand dispatch'' \cite{chehasmatbusmey17}.  It is believed that this paper will offer  new computational tools in this ongoing research.

\bibliographystyle{abbrv}

\bibliography{strings,KL_colt18}

\end{document}